\documentclass[12pt]{amsart}
\usepackage{verbatim}
\usepackage{amssymb}
\usepackage{enumerate}
\usepackage{graphicx}
\usepackage{psfrag}
\usepackage{stmaryrd}

\usepackage[utf8x]{inputenc}

\newtheorem{theorem}{Theorem}[section]
\newtheorem{proposition}[theorem]{Proposition}
\newtheorem{lemma}[theorem]{Lemma}
\newtheorem{corollary}[theorem]{Corollary}
\newtheorem{fact}[theorem]{Fact}

\newtheorem{claim}[theorem]{Claim}

 \newtheorem{dcc}{Double Cover Conjecture}  
\newtheorem{qlemma}{Lemma}

\theoremstyle{definition}

\newtheorem{definition}[theorem]{Definition}

\theoremstyle{remark}
\newtheorem{remark}{Remark}

\newcommand{\mc}[1]{\mathcal{#1}}
\newcommand{\mbb}[1]{\mathbb{#1}}

\newcommand{\mf}[1]{\mathfrak{#1}}

\newcommand{\setm}{\setminus}
\newcommand{\empt}{\emptyset}
\newcommand{\subs}{\subset}

\newcommand{\oo}{{{\omega}_1}}

\newcommand{\dom}{\operatorname{dom}}
\newcommand{\ran}{\operatorname{ran}}
\def\<{\left\langle}
\def\>{\right\rangle}
\def\cf{\operatorname{cf}}
\def\br#1;#2;{\bigl[ {#1} \bigr]^ {#2} }

\newcommand{\rrestriction}[1]{[#1]}
\newcommand{\ssetm}{\bbslash}

\author[L. Soukup]{Lajos Soukup}
\thanks
  {
   The author was supported by the
Hungarian National Foundation for Scientific Research grants no.
K 68262 and K 61600. }
\address
      { Alfr{\'e}d R{\'e}nyi Institute of Mathematics, Hungarian
        Academy of Sciences, Budapest, Hungary  }
\email{soukup@renyi.hu}
\urladdr{http://www.renyi.hu/$\sim$soukup}
\title[Elementary submodels]{
Elementary submodels in infinite combinatorics}
\date{\today}

\begin{document}

\begin{abstract}
The usage of elementary submodels
is a simple but powerful method to prove theorems, or to simplify proofs
in infinite combinatorics.
First
we  introduce all the necessary concepts of logic, then we 
prove classical theorems
using elementary submodels.
We also present a
 new proof of  Nash-Williams's theorem 
on cycle-decomposition of graphs, 
and finally we improve 
a decomposition theorem of Laviolette 
concerning  bond-faithful decompositions of
graphs.
\end{abstract}

\maketitle

\section{Introduction}

The aim of 
this paper is to explain how to use
elementary submodels  to prove new theorems or to simplify old proofs
in infinite combinatorics.
The paper mainly addresses novices
learning this  technique:
%
we  introduce all the necessary concepts and
give easy  examples to illustrate our method,
but the paper also
contains  new proofs of  theorems of Nash-Williams on
decomposition of infinite graphs, and an improvement of  a 
decomposition theorem
of Laviolette 
concerning bond-faithful decompositions.

The first 
known application of this method is due to Stephen G. Simpson, (see \cite{S} and the proof of 
\cite[Theorem 7.2.1]{CK}), who proved the Erdős-Rado Theorem using this technique, and 
indicated that  ``{\em one can give similar proofs for several other known theorems of combinatorial set theory ...''} 

Our aim is to popularize a {\em method} instead of giving 
just ``{\em black box}'' theorems.

In section \ref{sc:prelim} we recall and summarize all
necessary preliminaries from set theory, combinatorics and logic.

In section \ref{sc:first} 
we give the first application of elementary submodels, and we 
explain why it is natural to consider $\Sigma$-elementary submodels 
for  some large enough finite family $\Sigma$
of formulas.

In section \ref{sc:classical} we use elementary submodels to prove
some classical theorems in combinatorial set theory. All these
theorems
have the following Ramsey-like flavor:
{\em Every large enough structure contains large enough ``{nice}'' substructures.  }

In section \ref{sc:nw} we prove structure  theorems of a different kind: 
{\em Every large structure having certain properties can be
  partitioned into small ``nice'' pieces.}
A typical example is  Nash-Williams's theorem
on cycle decomposition of graphs without odd cuts.
To prove these structure theorems it is not enough to consider just
one
elementary submodel but we should introduce the concept of 
the {\em chains of elementary submodels}.

Finally, in section \ref{sc:bond},
we give a more elaborate application of chain of elementary submodels
to
eliminate GCH from a theorem concerning bond-faithful
decomposition of graphs.

\bigskip

This paper addresses persons
who are interested in infinite combinatorics, but who are not set theory specialist.
If you want to study 
more elaborated applications of these methods,
see the survey papers of Dow \cite{Dow} and Geshcke \cite{Geschke},
or  the book of Just and Weese \cite[Chapter 24]{Just-Weese}. 
These papers are highly more technical, than the current one, but they
 also contain many applications in set theoretic topology.

For applications of these methods in infinite combinatorics, 
see also \cite{BHT}, \cite{Hajnal}, \cite{HJSSz} and   \cite{KoSh}.
Chains of elementary submodels play also a crucial role in the proof of
some key results of the celebrated pcf theory of  Shelah,
see \cite{sh:pcf} or \cite{AbMa}. 


{
\sc 

%
%
%

%


%
%

}

\section{Preliminaries}\label{sc:prelim}

\subsection{Set theory}
We use the standard notions and notation  of set theory, see
\cite{Jech} or \cite{Kunen}. If
${\kappa}$ is a cardinal and $A$ is a set, let
\begin{equation}
\br A;<{\kappa};=\{a\subs A:|a|<{\kappa} \};\qquad
\br A;{\kappa};=\{a\subs A:|a|={\kappa} \}.
\end{equation}
If $X$ and $Y$ are sets let
$[X,Y]=\bigl\{\{x,y\}:x\in X, y\in Y\bigr\}$.

We denote by $V$ the class of all sets, and by
$On$  the class of all ordinals.
The {\em cumulative hierarchy} $\<V_{\alpha}:{\alpha}\in On\>$ is
defined by  transfinite induction on ${\alpha}$ as follows:
\begin{enumerate}[(1)]
\item $V_0=\empt$,
\item $V_{{\alpha}+1}=\mc P(V_{\alpha})$,
\item $V_{\beta}=\cup\{V_{\alpha}:{\alpha}<{\beta}\}$
if ${\beta}$ is a limit ordinal.
\end{enumerate}
\begin{fact}
$V=\cup \{V_{\alpha}:{\alpha}\in On\}$, i.e. for each set $x$
there is an ordinal ${\alpha}$ such that $x\in V_{\alpha}$.
\end{fact}

\subsection{Combinatorics}
We use the standard notions and notation  of combinatorics, see
e.g. \cite{Di}.
A {\em 
graph $G$}
is a pair $\<V(G),E(G)\>$, where $E(G)\subs \br V(G);2;$.
$V(G)$ and $E(G)$ are  the sets of {\em vertices and edges},
respectively,  of $G$. 
We always assume that $V(G)\cap E(G)=\empt$.

A {\em ${\kappa}$-cover}  of a graph $G$ is a family $\mc G$
of subgraphs of $G$ such that every edge of $G$ belongs to exactly
${\kappa}$ members of the family $\mc G$.
A {\em decomposition} is a 1-cover, i.e. a family $\mc G$ such that 
$\{E(G'):G'\in \mc G\}$ is  a partition of $E(G)$.

If $M$ is a set 
then let
\begin{equation}\notag
G\rrestriction M =\<V(G)\cap M, E(G)\cap \br M;2;\>;
\makebox[2mm]{} G\ssetm M=\<V(G), E(G)\setm M\>.
\end{equation}
So $G\ssetm M$ denotes the graph obtained from $G$ 
removing all edges in $M$.
If 
$\forall x,y  (x,y\in M \leftrightarrow \{x,y\}\in M) $, then 
the graphs $G[M]$ and $G\ssetm M$ form  a decomposition of  $G$.

If $G$ is fixed, and $A\subs V(G)$ then we write
$\bar A$ for $V(G)\setm A$.
A {\em cut} of $G$ is a set of edges of the form
$E(G)\cap [A,\bar A]$ for some $A\subs V(G)$.
A {\em bond} is a non-empty cut which is 
 minimal among the cuts
with respect to inclusion.

\begin{fact}\label{f:bond} 
$\empt\ne F\subs E(G)$  is a bond in $G$ iff
there are two distinct connected components $C_1$ and $ C_2$ of 
$G\ssetm F$ such that $F=E(G)\cap [C_1,C_2]$.
\end{fact}
%

The following statement will be used later several times.

\begin{proposition}\label{pr:bond}
Assume that $H$ is a subgraph of $G$, 
$F$ is a bond in $H$.
If $F$ is not a bond in $G$ then 
$F\subs \br D;2;$ for some connected component $D$ of $G\ssetm F$.
\end{proposition}

\begin{proof}
By Fact \ref{f:bond} 
there are two distinct connected components $C_1$ and $ C_2$ of 
$H\ssetm F$ such that $F=E(H)\cap [C_1,C_2]$.
If $C_1$ and $C_2$ are subsets of  different connected components
of $G$, $C_1\subs D_1$ and $C_2\subs D_2$, then 
\begin{equation}\notag
F=[C_1,C_2]\cap E(H)\subs [D_1,D_2]\cap E(G)\subs 
F\cup ([D_1,D_2]\cap E(G\setm F))=F,  
\end{equation}
i.e. 
$F=[D_1,D_2]\cap E(G)$ and so
$F$ is a bond in $G$ by Fact \ref{f:bond} above, which contradicts the assumptions. 
So  $C_1$ and $C_2$ are subsets of the same connected component
$D$ of $G\ssetm F$. Thus $F\subs [C_1, C_2]\subs \br D;2;$.
\end{proof}

Given a graph $G$ for  $x\ne y\in V(G)$ denote by ${\gamma}_G(x,y)$
the {\em edge connectivity} of $x$ and $y$ in $G$, i.e. 
\begin{equation}\notag
{\gamma}_G(x,y)=\min\{|F|:F\subs E(G):\text{$F$ separates $x$ and $y$
  in $G$}\}.  
\end{equation}
By the weak Erd\H os-Menger Theorem  there are 
${\gamma}_G(x,y)$ many edge disjoint paths between 
$x$ and $y$ in $G$.


\subsection{Logic}

The language of set theory is the first order language $\mc L$
containing only one binary relation symbol $\in$.
So the formulas of $\mc L$ are over the alphabet
$\{\lor, \neg,(,)\exists,=,\in\}\cup \mbox{Var}$, where $\mbox{Var}$
is an infinite set of variables.
To simplify our formulas we often use abbreviations like
$\forall x$, $\to$, $x \subs y$,
$\exists !x$, $\exists x\in y\ {\varphi}$, etc.

An {\em $\mc L$-structure} is a pair $\<M,E\>$, where
$E\subs M\times M$. In this paper we will consider only structures
in the form $\<M,\in\restriction M\>$
where $\in \restriction M$ is the restriction of the usual membership
relation to $M$, i.e. 
\begin{equation}\notag
\in \restriction M=\{\<x,y\>\in M\times M:x\in
y\}.  
\end{equation}
We usually write $\<M,\in\>$ or simply $M$ for $\<M,\in \restriction M\>$.

If ${\varphi}(x_1,\dots, x_n)$ is a formula,
$a_1,\dots, a_n$ are sets, then let
${\varphi}(a_1,\dots, a_n)$ be the formula obtained from
${\varphi}(x_1,\dots, x_n)$ by replacing each free
occurrence of $x_i$ with $a_i$. [An occurrence of $x_i$ is {\em free} it is not
  within the scope of a quantifier $\exists x_i$.]

If ${\varphi}(x,x_1,\dots, x_n)$ is a formula, $a_1,\dots, a_n$ are
sets,
then $C=\{a: {\varphi}(a,a_1,\dots, a_n)\}$ is a {\em class}.
Especially, every set $b$ is a class: $b=\{a:a\in b\}$.
Moreover, all sets form the class $V$:
$V=\{a:a=a\}$. In this paper we will consider just these classes: the
sets and the ``universal'' class $V$.

For a formula ${\varphi}(x_1,\dots, x_n)$, a class $M$, and for
$a_1,\dots, a_n\in M$ we define when
\begin{equation}
M\models {\varphi}(a_1,\dots, a_n),
\end{equation}
i.e. when {\em $M$ satisfies ${\varphi}(a_1,\dots, a_n)$}, by
induction on the complexity of the formulas
in the usual way: 
\begin{enumerate}[(i)]
\item $M\models$ ``$a_i\in a_j$'' iff $a_i\in a_j$,
\item $M\models$ ``${\varphi}\lor \psi$'' iff
$M\models {\varphi}$ or $M\models \psi$.
\item $M\models$ ``$\neg {\varphi}$'' iff $M\models {\varphi}$ fails.
\item $M\models $ ``$\exists x {\varphi}(x, a_1,\dots a_n)$''
iff there is an $a\in M$ such that
$M\models$ ``${\varphi}(a,a_1,\dots, a_n)$''
\end{enumerate}
For a formula ${\varphi}(x_1,\dots, x_n)$
let ${\varphi}^M(x_1,\dots,x_n)$ be the formula obtained
by replacing each quantifier $\exists x$ with
$\exists x\in M$ in ${\varphi}$.
Clearly for each $a_1,\dots, a_n\in M$,
\begin{equation}
{\varphi}^M(a_1,\dots, a_n)\text{ iff } M\models
{\varphi}(a_1,\dots, a_n).
\end{equation}

If ${\varphi}(x_1,\dots,x_n)$ is a formula, $M$ and $N$ are classes,
$M\subs N$, then we say that  ${\varphi}$ is {\em absolute between
$M$ and $N$,}
\begin{equation}
 M\prec_{\varphi}N
\end{equation}
in short, iff for each $a_1,\dots, a_n\in M$
\begin{equation}
 M\models {\varphi}(a_1,\dots,a_n)
\text{ iff } N\models{\varphi}(a_1,\dots,a_n)
\end{equation}
If $\Sigma$ is a collection of formulas then write
\begin{equation}
M\prec_\Sigma N
\end{equation}
iff $M\prec_{\varphi} N$ for each ${\varphi}\in \Sigma$.

$M$ is an {\em elementary submodel} of $N$,
\begin{equation}
 M\prec N
\end{equation}
iff $M\prec_{\varphi} N$ for each formula ${\varphi}$.

If ${\varphi}$ is absolute between $M$ and $V$, 
then we say that ${\varphi}$ is {\em absolute for $M$}.

\begin{theorem}[L\"owenheim-Skolem]\label{tm:ls}
For each set $N$ and infinite subset $A\subs N$ there is a set $M$
such that $A\subs M\prec N$ and $|M|=|A|$.
\end{theorem}

Since $ZFC\not\,\vdash Con(ZFC)$ by 
G\"odel's Second Incompleteness Theorem, it is not provable in 
ZFC that there is a set $M$ with $M\models ZFC$.
So, since $V\models ZFC$, it is not provable in ZFC that there is a set
$M$  with $M\prec V$. Thus, in the L\"owenheim-Skolem theorem above, the
assumption that $N$ is a set was essential.
However, as we will see,  the following result can serve as a substitute
for the L\"owenheim-Skolem theorem for classes in certain cases. 

\begin{theorem}[Reflection Principle]\label{pr:refl}
Let $\Sigma$ be a finite collection of formulas. Then for each
cardinal ${\kappa}$ there is a cardinal ${\lambda}$ such that
$V_{\lambda}\prec_\Sigma V$, and 
$\br V_{\lambda};<{\kappa};\subs V_{\lambda}$.
\end{theorem}

We need some corollaries of this theorem. Let us recall that 
the {\em cofinality 
$\cf(\alpha)$} of an ordinal $\alpha $  is 
the least of the cardinalities of the cofinal subsets of $\alpha$.
A cardinal $\kappa$ is {\em regular} iff $\kappa=\cf(\kappa)$.

\begin{corollary}\label{cor:elem}
Let $\Sigma$ be a finite collection of formulas, ${\kappa}$ an infinite
cardinal, and $x$ a set.\\
(1) There is a set $M\prec_\Sigma V$ with $x\in M$ and $|M|={\kappa}$.\\
(2) If ${\kappa}>{\omega}$ is regular then there is a set  $M\prec_\Sigma V$ with
$x\in M$, $|M|<{\kappa}$ and $M\cap {\kappa}\in {\kappa}$.\\
(3) If ${\kappa}^{\omega}={\kappa}$ then there is
a set $M\prec_\Sigma V$ such that
$x\in M$, $|M|={\kappa}$, $M\cap {\kappa}^+\in {\kappa}^+$, and $\br
M;{\omega};\subs M$. \\
(4) If ${\kappa}>{\omega}$ is regular then the set
\begin{equation}\notag
S_x=\{M\cap {\kappa}:x\in M\prec_\Sigma V, M\cap {\kappa}\in {\kappa}\}
\end{equation}
contains a closed unbounded subset of ${\kappa}$.
\end{corollary}

\begin{proof}
Fix a cardinal $\mu\ge \kappa$ with $x\in V_{\mu}$.
By the Reflection Principle there is a cardinal ${\lambda}>\mu$
such that $V_{\lambda}\prec_\Sigma V$ and $\br
V_{\lambda};{\kappa};\subs V_{\lambda}$.\\

\noindent
(1) Straightforward from  the L\"owenheim-Skolem
  theorem: since $V_{\lambda}$ is a set, $|V_{\lambda}|\ge {\kappa}$,
and $x\in V_{\lambda}$
there is $M\prec V_{\lambda}$ with $x\in M$ and $|M|={\kappa}$.
Then $M\prec_\Sigma V$.

\medskip
\noindent
(2)
Construct a sequence $\<M_n:n<{\omega}\>$ of elementary submodels
 of $V_{\lambda}$ with $|M_n|<{\kappa}$ as follows.
Let $M_0$ be a countable elementary submodel of
$V_{\lambda}$ with $x\in M$.
If $M_n$ is constructed, let
${\alpha}_n=\sup (M_n\cap {\kappa}) $.
Since ${\kappa}$ is regular we have ${\alpha}_n<{\kappa}$.
By the L\"owenheim-Skolem theorem there is an elementary submodel
$M_{n+1}$ of $V_{\lambda}$ such that
 $M_n\cup {\alpha}_n\subs M_{n+1}$ and
$|M_{n+1}|=|M_n\cup {\alpha}_n|<{\kappa}$.
Finally let $M=\cup \{M_n:n<{\omega}\}$.
Then $M\prec V_{\lambda}$, and so
$M\prec_\Sigma V$,  and $M\cap {\kappa}=\sup {\alpha}_n\in {\kappa}$.

\medskip
\noindent
(3)
Construct an increasing  sequence 
$\<M_{\nu}:{\nu}<\oo\>$
of elementary submodels
 of
$V_{\lambda}$ with $|M_{\nu}|={\kappa}$ as follows.
Let $M_0$ be an  elementary submodel of
$V_{\lambda}$ with ${\kappa}\cup \{x\}\subs M_0$ and $|M_0|={\kappa}$.
For limit ${\nu}$ let $M_{\nu}=\cup\{M_{\beta}:{\beta}<{\nu}\}$.
If $M_{\nu}$ is constructed, let
${\alpha}_{\nu}=\sup (M_{\nu}\cap {\kappa}^+) $.
Since $|M_{\nu}|={\kappa}$ we have ${\alpha}_{\nu}<{\kappa}^+$.
Let $X_{\nu}=M_{\nu}\cup {\alpha}_{\nu}\cup \br M_{\nu};{\omega};$.
Then $|X_{\nu}|\le {\kappa}^{\omega}={\kappa}$.
By the L\"owenheim-Skolem theorem there is an elementary submodel
$M_{{\nu}+1}$ of $V_{\lambda}$ with $X_{\nu}\subs M_{\nu+1}$ and $|M_{{\nu}+1}|={\kappa}$.
Finally let $M=\cup \{M_{\nu}:{\nu}<\oo\}$.
Since ${\kappa}\ge {\omega}_1$,  $M\cap {\kappa}^+=\sup \{{\alpha}_{\nu}:{\nu}<\oo\}\in{\kappa}^+$.
If $A\in \br M;{\omega};$ then there is ${\nu}<\oo$
with $A\subs M_{\nu}$, and so $A\in X_{\nu}\subs  M_{{\nu}+1}\subs M$.

\medskip
\noindent
(4)
Construct a continuous increasing chain of elementary submodels
$\<M_{\nu}:{\nu}<{\kappa}\>$ of
$V_{\lambda}$ with $|M_{\nu}|\le {\nu}+{\omega}$ as follows.
Let $M_0$ be a countable   elementary submodel of
$V_{\lambda}$ with $x\in M$.
For limit ${\nu}$ let $M_{\nu}=\cup\{M_{\beta}:{\beta}<{\nu}\}$.
If $M_{\nu}$ is constructed, let
${\alpha}_{\nu}=\sup (M_{\nu}\cap {\kappa}^+) $.
Since $|M_{\nu}|<{\kappa}$ and ${\kappa} $ is regular
we have ${\alpha}_{\nu}<{\kappa}$.
Let $X_{\nu}=M_{\nu}\cup ({\alpha}_{\nu}+1)$.
Since $|X_{\nu}|\le {\nu}+{\omega}$,
by the L\"owenheim-Skolem theorem there is an elementary submodel
$M_{{\nu}}$ of $V_{\lambda}$ with $X_{\nu}\subs M_{\nu}$ and
$|M_{\nu}|=|X_{\nu}|$.

Then $C=\{{\alpha}_{\nu}:{\nu}<{\kappa}\}$ is a closed
unbounded  subset of
${\kappa}$ and $C\subs S_x$ because ${\alpha}_{\nu}\in S_x$ is witnessed
by $M_{\nu}$.
\end{proof}

\subsection{Absoluteness}
A set $b$ is {\em definable from parameters $a_1,\dots, a_n$} 
 iff there is a formula
${\varphi}(x)$ such that
\begin{equation}
\forall x ({\varphi}(x,a_1,\dots, a_n)\leftrightarrow x=b).
\end{equation}
We say that $b$ is {\em definable}  iff 
we do not need any parameters, i.e.
$\forall x ({\varphi}(x)\leftrightarrow x=b)$.

\begin{claim}\label{cl:def}
If $b$ is definable from the parameters $a_1, \dots a_n\in M$ by the formula 
${\varphi}(x, \vec y)$, and
$M\prec_{\{\exists x {\varphi}(x, \vec y), {\varphi}(x, \vec y)\}}V$, then $b\in M$.
\end{claim}

\begin{proof}
Since $M\prec_{\exists x {\varphi}(x,\vec y)}V$, $\vec a\in M$ and so
$M\models \exists x {\varphi}(x,\vec a)$, there is $b'\in M$
such that $M\models {\varphi}(b',\vec a)$. Thus
$M\prec_{\varphi(x,\vec y)} V$
yields $V\models {\varphi}(b',\vec a)$, and so $b=b'\in M$.
\end{proof}

Given a class $N$
we say that a formula ${\varphi}(x_1,\dots, x_n,y)$ {\em defines the  operation
$F^N_{\varphi}$ in $N$} iff $N\models\forall x_1,\dots, ,x_n\exists! y
{\varphi}(x_1,\dots, x_n,y)$,
and for each $a_1,\dots, a_n,b\in N$, $F^N_{\varphi}(a_1,\dots, a_n)=b$ iff
$N\models{\varphi}(a_1,\dots, a_n, b)$.
If $V=N$ then we omit the superscript $V$.

Given a class $N$
we say that the operation $F_{\varphi}$ is {\em absolute for $N$} provided
 ${\varphi}$ defines an  operation in $N$,
and ${\varphi}(\vec x,y)$  is absolute for $N$.

\begin{claim}\label{cl:op}
If the formula ${\varphi}$ defines the operation $F_{\varphi}$ in $V$,
and we have $M\prec_{\{\forall \vec x\exists y {\varphi}(\vec x,y), {\varphi}(\vec x,y)\}} V$,
then  ${\varphi}$ defines an operation $F^M_{\varphi}$ in $M$,
and $F^M_{\varphi}=F_{\varphi}\restriction M$.
\end{claim}

\begin{proof}
Since $M\prec_{\forall \vec x\exists y {\varphi}(\vec x,y)}V$,
for each $a_1,\dots, a_n\in M$ there is $b\in M$ such that
$M\models {\varphi}(\vec a,b)$. Thus $V\models {\varphi}(\vec a,b)$, and so
$F_{\varphi}(\vec a)=b\in M$.
If $M\models {\varphi}(\vec a,b)\land {\varphi}(\vec a,b')$
then $V\models {\varphi}(\vec a,b)\land {\varphi}(\vec a,b')$, so $b=b'$.
Thus $M\models \forall \vec x\exists ! y {\varphi}(\vec x,y)$.
\end{proof}

\section{First application of elementary submodels.}
\label{sc:first}
In this section we present an example 
\begin{itemize}
\item 
to  illustrate our basic method,
\item
to 
indicate
the main technical problem of this approach;
and also
\item  to give
a solution to that technical problem.
\end{itemize}

In \cite{nw} Nash-Williams proved that
 {\em a graph $G$ is decomposable into cycles if and only if
it has no odd cut.}
In Section \ref{sc:nw} we give a new proof of this result.
Let us say that a  graph $G$ is {\em NW} iff it does not have any odd
cut.
We will prove the Nash-Williams Theorem by induction on $|V(G)|$.
Since the statement is trivial for countable graphs, it is enough
to decompose an uncountable  NW-graph $G$ into  NW-graphs of
smaller cardinality.
We will use ``small'' elementary submodels to cut  the graph $G$
into the right pieces.
To do so we need two lemmas, the first (and easy)
one will serve as the first example of the application
of our method.

\newcommand{\kkappa}{W}

First we assume that we could work with a full elementary submodel of $V$,
and we   discuss later how to get around the technical 
difficulties that arise in this naive approach.

\begin{lemma}\label{lm:weak}
If $G=\<{\kkappa},E\>$ is an NW-graph,
$G\in M\prec V$, then $G\rrestriction M=G[M\cap {\kkappa}]$ is also an NW-graph.
\end{lemma}
%

\begin{proof}
Assume on the contrary that  $G\rrestriction M$ has an odd cut
$F=\{f_1,\dots, f_{2n+1}\}$. Since any cut is the disjoint union of
bonds  we can assume that
$F$ is a bond.
Since $F$ can not be a bond in $G$, by
Proposition \ref{pr:bond} there is a connected component
$D$ of $G\ssetm F$ such that $F\subs \br D;2;$.
Let $bc\in F$.
Then $b$ and $c$ are in $D$, $D$ is connected, so 
there is a path
$bw_1w_2\dots w_{m-1}c$ between $b$ and $c$ in $G$ which avoids $F$.

\begin{claim}\label{cl:omegax}
$\br M;<{\omega};\subs M$.
\end{claim}

\begin{proof}[Proof of the claim]
Consider the operations $F_1(x,y)=\{x,y\}$ and
$F_2(z)=\cup z$. By Claim \ref{cl:op}, there are formulas
${\sigma}_1, {\sigma}'_1, {\sigma}_2$ and ${\sigma}'_2$
such that if $N \prec_{\{{\sigma}_i,{\sigma}'_i\}}V$ then $N$ is closed
under operation $F_i$, $i=1,2$.

Since $M\prec V$, this yields that $M$ is closed under $F_1$ and $F_2$.
Since
\begin{equation}
\{a_0,\dots, a_{n}\}=
\cup\{\{a_0,\dots, a_{n-1}\},\{a_n\}\}
\end{equation}
we obtain $\br M;<{\omega};\subs M$ by induction on $n$.
\end{proof}

\begin{claim}\label{ck:osub}
${\omega}\cup\{{\omega}\}\subs M$.
\end{claim}

\begin{proof}[Proof of the Claim]
$\empt$ and ${\omega}$ are definable, so by Claim \ref{cl:def}
there are formulas ${\rho}_1$ and ${\rho}_1'$, and 
${\rho}_2$ and ${\rho}_2'$, respectively, such that
if $N\prec_{\{{\rho}_1,{\rho}'_1\}}V$ then $\empt\in N$, and if
$N\prec_{\{{\rho}_2,{\rho}'_2\}}V$ then ${\omega}\in N$.
Since $M\prec V$, this implies $\empt, {\omega}\in M$.

Consider the operation $F_3(x)=x\cup\{x\}$.
By Claim \ref{cl:op}, there are formulas
${\sigma}_3 $ and  $ {\sigma}'_3$
such that if $N \prec_{\{{\sigma}_3,{\sigma}'_3\}}V$ then $N$ is closed
under operation $F_3$.
Since $M\prec V$, this yields that $M$ is closed under $F_3$.
So $0\in M$ and $n+1=F_3(n)$ imply ${\omega}\subs M$.
\end{proof}

So  we have $F\in M$ and $m\in  M$.
Consider the following formula
${\varphi}_1(G,m,f,b,c,F)$:
\begin{multline}
\text{$G$ is a graph, 
$f$ is a function,
$\dom(f)=m$, $\ran(f)\subs V(G)$,}\\
\text{$f(0)=b, f(m-1)=c\land
(\forall i<m-1)\ \{f(i),f(i+1)\}\in E(G)\setm F.$}
\end{multline}
Since
\begin{equation}\label{f1}
\exists f\ {\varphi}_1(G,m,f,b,c,F),
\end{equation}
the assumption $M\prec_{\exists f {\varphi_1}(G,m,f,b,c,F)} V$  and $G,m,b,c,F\in M$ imply that
the same formula holds in $M$.
So there is $f\in M$
such that
\begin{equation}\label{f2}
{\varphi}_1(G,m,f,b,c,F).
\end{equation}
Since $M\prec_{{\varphi}_1(G,m,f,b,c,F)}  V$ we have
\begin{equation}
{\varphi}_1(G,m,f,b,c,F).
\end{equation}
To complete the proof we need one more claim.
\begin{claim}\label{cl:eval}
If $g\in M$ is a function, $x\in \dom(g)$,
then $g(x)\in M$.
\end{claim}
\begin{proof}[Proof of the Claim]
Consider the evaluation operation
$F_4(g,y)=g(y)$.
By Claim \ref{cl:op}, there are formulas
${\sigma}_4$ and $ {\sigma}'_4$
such that if $N \prec_{\{{\sigma}_4,{\sigma}'_4\}}V$ then $N$ is closed
under operation $F_4$.
Since $M\prec V$, this yields that $M$ is closed under the evaluation
operation $F_4$.
\end{proof}

By  Claim  \ref{cl:eval} above,
$\ran(f)\subs M\cap {\kkappa}$,   and so
 $f(0)f(1)\dots f(m-1)$ is a path between $b$ and $c$ in
$G\rrestriction M$
which avoids $F$. Contradiction.
\end{proof}

So if $M$ is a ``small'' elementary submodel of $V$, then
$G\rrestriction M$ is a ``small'' NW-subgraph of $G$.
 Unfortunately, as we explained before the formulation of the
 Reflection Principle,
 we can not get any set  $M$ with $M\prec V$ 
 by the Second Incompleteness Theorem of G\"odel.
So we can not apply the lemma above to prove the Nash-Williams Theorem.

Fortunately, this is just a technical problem because  one can observe
that 
in the proof above we have not used the full
power of $M\prec V$,
we applied the absoluteness only for finitely many
formulas between $V$ and $M$.
Namely, we used only the absoluteness for the formulas from the family
\begin{equation}
\Sigma^*= \{{\sigma}_i,{\sigma}'_i,: i=1,2,3,4\}
\cup \{{\rho}_j,{\rho}'_j,: j=1,2\}
\cup\{
\exists f {\varphi}_1,{\varphi}_1\}.
\end{equation}

So actually the proof of Lemma \ref{lm:weak} yields the following result:
\begin{lemma}\label{lm:nw1}
If $G=\<{\kkappa},E\>$ is an NW-graph,
$G\in M\prec_\Sigma  V$ for some large enough finite set $\Sigma$ of
formulas,
then $G\rrestriction M$ is also an NW-graph.
\end{lemma}

In many proofs
 we will  argue in the following way:
 \begin{enumerate}[(I)]
 \item using the Reflection Principle  we can find a  cardinal ${\lambda}$
such that $V_{\lambda}$ resembles $V$ in two ways:
\begin{enumerate}[(1)]
\item $\br V_{\lambda};{\kappa};\subs V_{\lambda}$ for some large enough
cardinal ${\kappa}$, and
\item  $V_{\lambda}\prec_\Sigma V$ for some large enough finite collection
$\Sigma$ of formulas.
\end{enumerate}
 \end{enumerate}
We can not use the model $V_{\lambda}$ directly, because it is too large, but
 \begin{enumerate}[(I)]\addtocounter{enumi}{1}
\item since $V_{\lambda}$ is a set, we can use the L\"owenheim-Skolem Theorem
to find a small elementary submodel $M$ of $V_{\lambda}$ which
contains $G$.
 \end{enumerate}
 Then $M\prec_\Sigma V$.

We do not fix $\Sigma$ in advance. Instead of this
we write down the proof, and after that
we put all the formulas for which we used the absoluteness
into $\Sigma$.
Actually, apart from the proof of Lemma \ref{lm:nw1} above,
 we will not construct $\Sigma$ explicitly.

\begin{remark}
We will show later that if $\Sigma$ is large enough then
$G\ssetm M$ is also an NW-graph,
so the pair $\<G\rrestriction M,G\ssetm M\>$ is a decomposition 
of $G$ into
NW-graphs.
\end{remark}

\subsection{More on absoluteness}
In Claim \ref{cl:omega} below we summarize certain observations 
we made in the proof of
Lemma \ref{lm:weak} above.

\begin{claim}\label{cl:omega}
There is a finite collection $\Sigma_0$ of formulas such that if
$M\prec_{\Sigma_0} V$ then
$\br M;<{\omega};\subs M$, ${\omega}\cup\{{\omega}\}\subs M$,
and $f(x)\in M$ for each function $f\in M$ and $x\in \dom(f)\cap M$.
\end{claim}

We need two more easy claims.

\begin{claim}\label{cl:card}
There is a finite collection $\Sigma_1$ of formulas such that
if $M\prec_{\Sigma_1} V$ then for each $A\in M$ if
$|A|\subs M$ then $A\subs M$.
\end{claim}

\begin{proof}
Let $\Sigma_1\supset \Sigma_0$ be a finite family of formulas such that
\begin{enumerate}
\item the formulas ``{\em $f$ is a bijection between $x$ and $y$}''
and ``{\em $\exists f$ ($f$ is a bijection between $x$ and $y$)}''
are in $\Sigma_1$,
\item if $M\prec_{\Sigma_1} V$ then $M$
is closed under the ``cardinality'' operation $A\mapsto |A|$.
\end{enumerate}

Assume that $|A|={\kappa}$.
Then ${\kappa}\in M$ by $(2)$.
Since $V\models ``${\em $\exists f $ $f$ is a bijection between
  ${\kappa}$  and $A$}''
 there is $f\in M$ such that
$M\models $ ``{\em $f$ is a bijection from ${\kappa}$ onto $A$}''.
Then $f$ is a bijection from ${\kappa}$ to $A$ by (1).
So if $a\in A$ then there is ${\alpha}\in {\kappa}$
such that $f({\alpha})=a$. 
We assumed that $|A|\subs M$, so $\alpha\in M$ as well.
Thus $f,{\alpha}\in M$ implies
 $f({\alpha})\in M$ by $\Sigma_1\supset \Sigma_0$. Thus $A\subs M$.
\end{proof}

\begin{claim}\label{cl:omegasub}
If $M\prec_{\Sigma_0\cup \Sigma_1} V$ then for each countable 
set $A\in M$ we have $A\subs M$.
\end{claim}

\begin{proof}
If $A$ is countable then $|A|={\omega}\subs M$ by Claim \ref{cl:omega}
 because $M\prec_{\Sigma_0}V$.
Thus $A\subs M$ by Claim \ref{cl:card} because $M\prec_{\Sigma_1}V$.
\end{proof}

\section{Classical theorems}\label{sc:classical}
In this section we prove some classical theorems using elementary submodels.
The Erdős-Rado Theorem was proved by
 Stephen G. Simpson, (see \cite{S} and \cite[Theorem 7.2.1]{CK}) using this technique,
and  for the late seventies the method became widely known among the set theory specialists,
so  the other proofs in this section are all from the folklore. 

\medskip

A family $\mc A$ is called a {\em $\Delta$-system with kernel $D$}
iff $A\cap A'=D$ for each $A\ne A'\in \mc A$.
A {\em $\Delta$-system} is a $\Delta$-system with some kernel.

\begin{theorem}
Every uncountable family $\mc A$
of finite sets contains an uncountable
$\Delta$-system.
\end{theorem}

\begin{proof}
 We can assume that $\mc A \subs \br \oo;<{\omega};$.

Let $\Sigma$ be a large enough finite set of formulas.
By Corollary \ref{cor:elem}(1) there is a countable set $M$
such that  $\mc A\in M\prec_\Sigma V$.

Since $\mc A$ is uncountable, we can pick $A\in \mc A\setm M.$
Let $D=M\cap A$. Since $\br M;<{\omega};\subs M$ we have
$D\in M$ by Claim \ref{cl:omega}.
Let
\begin{equation}
\mbb B=\{\mc B\subs \mc A: \mc B\text{ is a $\Delta$-system with
  kernel $D$}\}.
\end{equation}
Since $\mc A, D\in M$ we have $\mbb B\in M$ as well. Moreover,
\begin{equation}
\exists \mc B\
\text{($\mc B$ is a $\subs$-maximal element of $\mbb B$).}
\end{equation}
Since $M\prec_{\Sigma}V$,
and the parameter $\mbb B$ is in $M$,
there is $\mc B\in M$ such that
\begin{equation}
M\models
\text{
($\mc B$ is a $\subs$-maximal element of $\mbb B$).}
\end{equation}
Since $M\prec_{\Sigma}V$,
we have
\begin{equation}
\text{
$\mc B$ is a $\subs$-maximal element of $\mbb B$.}
\end{equation}

\noindent{\bf Claim:} {\em $\mc B$ is uncountable}.\\
Assume on the contrary that $\mc B$ is countable.
Then, by claim \ref{cl:omegasub}, $M\prec_{\Sigma}V$
implies $\mc B\subs M$. Let  $\mc C=\mc B\cup\{A\}$.
Since $A\notin M$, $\mc C\supsetneq \mc B$.
 If $B\in \mc B$, then $B\in M$ and so $B\subs M$
and  $D\subs A\cap B\subs A\cap M=D$.
 So $\mc C\supsetneq \mc B$
is a $\Delta$-system with kernel $D$, i.e.
$\mc B$  was not a $\subs$-maximal element of $\mbb B$. This
contradiction
proves the claim.
\end{proof}

\begin{remark}
In each proof of this  section we will argue in the following way.
Let $\mc A$ be a structure of ``size'' ${\kappa}$.
Let $M\prec_\Sigma V$ for some large enough finite family
$\Sigma$ of formulas with $\mc A\in M$ and $|M|<{\kappa}$,
i.e. $M$ is a ``small'' elementary submodel which contains, as
an element,  a ``large'' structure $\mc A$. Since
$M$ has less elements than the size of $\mc A$, there is $A$ from $\mc
A$ such that $A\notin M$.  Then this $A$ has some ``{trace}''
$D$ on $M$. If $M$ is ``closed enough'' then this trace $D$ is in $M$.
Using this trace we define, in $M$, a maximal, ``nice''
substructure $\mc B$ of $\mc A$.
Then, using the fact that $A\notin M$, we try to prove that
$\mc B$ is large ``enough''.
\end{remark}

In the proof above we could use an arbitrary countable elementary submodel $M$ of
$V_{\lambda}$ with $\mc A\in M$.
However, in the next proof we need 
elementary submodels with some extra properties.

\begin{theorem}
If $\mc A$ is a family of finite sets such that
${\kappa}=|\mc A|$ is an  uncountable regular cardinal, then
$\mc A$ contains a
$\Delta$-system of size ${\kappa}$.
\end{theorem}

\begin{proof}
We can assume that $\mc A \subs \br {\kappa};<{\omega};$.

Let $\Sigma$ be a large enough finite set of formulas.
By Corollary \ref{cor:elem}(2) there is a set $M$ with $|M|<{\kappa}$
such that  $\mc A\in M\prec_\Sigma V$ and $M\cap {\kappa}\in {\kappa}$.

Since $|\mc A|={\kappa}$, we can pick $A\in \mc A\setm M.$
Let $D=M\cap A$. Since $\br M;<{\omega};\subs M$ we have
$D\in M$ by Claim \ref{cl:omega}.
Then
\begin{equation}
\exists \mc B\
\text{($\mc B\subs \mc A$ is $\subs$-maximal among
the  $\Delta$-systems with kernel $D$).}
\end{equation}
Since $M\prec_{\Sigma}V$,
and the parameters $\mc A$ and $D$ are in $M$,
there is $\mc B\in M$ such that
\begin{equation}
M\models
\text{($\mc B\subs \mc A$ is $\subs$-maximal among
the  $\Delta$-systems with kernel $D$).}
\end{equation}
Since $M\prec_{\Sigma}V$,
\begin{equation}
\text{$\mc B\subs \mc A$ is $\subs$-maximal among
the  $\Delta$-systems with kernel $D$.}
\end{equation}

\noindent{\bf Claim:} {\em $|\mc B|={\kappa}$}.\\
Assume on the contrary that $|\mc B|<{\kappa}$.
Since $\mc B\in M$ we have
$|\mc B|\in M\cap {\kappa}$. Thus $|\mc B|\subs M$ and so
$\mc B\subs M$ by Claim \ref{cl:card}.

 Let  $\mc C=\mc B\cup\{A\}$.
If $B\in \mc B$, then $B\in M$ and so $B\subs M$
by $M\prec_{\Sigma}V$. Thus $B\cap A=D$. So $\mc C\supsetneq \mc B$
is a $\Delta$-system with kernel $D$. Contradiction.
\end{proof}

\newcommand{\cp}{\mf c^+}

To prove the next theorem we need elementary submodels with one more
additional property.
\newcommand{\kp}{{\kappa}^+}

\begin{theorem}
If ${\kappa}^{\omega}={\kappa}$ then
every family  $\mc A=\{A_{\alpha}:{\alpha}< \kp\}\subs
\br \kp;{\omega};$ contains a $\Delta$-system of size
$\kp$.
Especially, 
every family  $\mc A=\{A_{\alpha}:{\alpha}< \cp\}\subs
\br \cp;{\omega};$ contains a $\Delta$-system of size
$\cp$.
\end{theorem}

\begin{proof}
Let $\Sigma$ be a large enough finite set of formulas.
By Corollary \ref{cor:elem}(3) there is a set $M$ with $|M|={\kappa}$
such that  $\mc A\in M\prec_\Sigma V$, $M\cap {\kappa}^+\in
{\kappa}^+$
and $\br M;{\omega};\subs M$.

Since $|\mc A|=\kp>|M|$,
we can pick $A\in \mc A\setm M.$
Let $D=M\cap A$. Then $D\in \br M;\le {\omega};$.
Since $\br M;<\omega;\subs M$ by Claim \ref{cl:omega}, 
and we assumed   $\br M;{\omega};\subs M$, 
we have $D\in M$.

Then
\begin{equation}
\exists \mc B\
\text{($\mc B\subs \mc A$ is $\subs$-maximal among
the  $\Delta$-systems with kernel $D$).}
\end{equation}
Since $M\prec_{\Sigma}V$ and $\br M;{\omega};\subs M$,
 the parameters $\mc A$ and $D$ are in $M$, so
there is $\mc B\in M$ such that
\begin{equation}
M\models
\text{($\mc B\subs \mc A$ is $\subs$-maximal among
the  $\Delta$-systems with kernel $D$).}
\end{equation}
Since $M\prec_{\Sigma}V$,
\begin{equation}
\text{$\mc B\subs \mc A$ is $\subs$-maximal among
the  $\Delta$-systems with kernel $D$.}
\end{equation}

\noindent{\bf Claim:} {\em $|\mc B|=\kp$}.\\
Assume on the contrary that $|\mc B|\le{\kappa}$.
 Thus $|\mc B|\subs {\kappa}\subs M$ and so
$\mc B\subs M$ by Claim \ref{cl:card}.

Let  $\mc C=\mc B\cup\{A\}$.
If $B\in \mc B$, then $B\in M$ and so $B\subs M$
and $A\cap B=D$
by $M\prec_{\Sigma}V$. So $\mc C\supsetneq \mc B$
is a $\Delta$-system with kernel $D$. Contradiction.
\end{proof}

Next we prove two classical partition theorems.
First we recall (a special case of) the arrow notation notation of Erdős and Rado.
Assume that  $\alpha,\beta$ and $\gamma$ ordinals. We write 
\begin{equation}
\alpha\to(\beta,\gamma)^2 
\end{equation}
iff given any function $f:\br \alpha;2;\to 2$ 
either there is a subset $B\subs \alpha$ of order type 
$\beta$ with $f''\br B;2;=\{0\}$, or 
there is a subset $C\subs \alpha$ of order type 
$\gamma$ with $f''\br C;2;=\{1\}$.

\begin{theorem}[Erd\H os--Dusnik--Miller]
If ${\kappa}=\cf({\kappa})>{\omega}$ then
${\kappa}\to ({\kappa},{\omega}+1)^2$.
\end{theorem}

\begin{proof}
Fix a coloring  $f:\br {\kappa};2;\to 2.$

Let $\Sigma$ be a large enough finite set of formulas.
By Corollary \ref{cor:elem}(2) there is a set $M$ with $|M|<{\kappa}$
such that  $f\in M\prec_\Sigma V$ and $M\cap {\kappa}\in {\kappa}$.

\begin{figure}[h]
\psfrag*{0}{$0$}
\psfrag*{1}{$1$}
\psfrag*{A}{$A$}
\psfrag*{B}{$B$}
\psfrag*{C}{$C$}
\psfrag*{M}{$M$}
\psfrag*{xi}{${\xi}$}
\psfrag*{be}{${\beta}$}
\psfrag*{ga}{${\gamma}$}
\psfrag*{ka}{${\kappa}$}
\includegraphics[width=\textwidth]{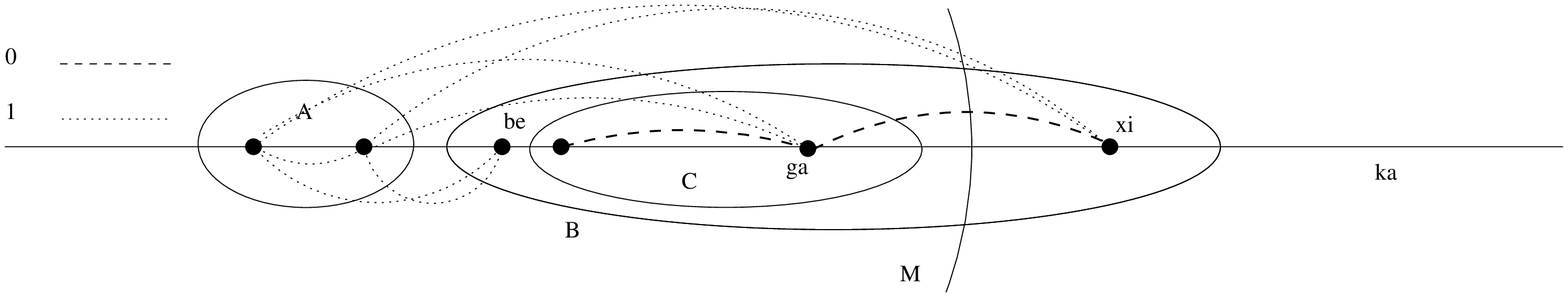}
\caption{}
\end{figure}

Fix ${\xi}\in {\kappa}\setm M$.
Let $A$ be a $\subs$-maximal subset of $M\cap {\kappa}$
such that $A\cup \{{\xi}\}$ is $1$-homogeneous.
If $A$ is infinite, then we are done.

Assume that $A$ is finite.
Let
\begin{equation}
B=\{{\beta}\in {\kappa}\setm A:\forall {\alpha}\in A\
f({\beta},{\alpha})=1\}. \end{equation}
Clearly ${\xi}\in B$.
Since $f,A\in M$ we have  $B\in M$.
Let $C\subs B$ be a $\subs$-maximal $0$-homogeneous subset.

\noindent{\bf Claim:} {\em $|C|={\kappa}$}.\\
Assume on the contrary that  $|C|<{\kappa}$.
Then $|C|\in M\cap {\kappa}$ and so $|C|\subs M$ because
$M\cap {\kappa}\in {\kappa}$.
Thus $C\subs M$
by Claim \ref{cl:card}.
Let ${\gamma}\in C$. Since ${\gamma}\in M\setm A$ we have that
$A\cup \{{\gamma}\}\cup \{{\xi}\}$ is not $1$-homogeneous.
But $A\cup \{{\xi}\}$ is  $1$-homogeneous and ${\gamma}\in B$,
so   $f({\gamma},{\xi})=0$.
Thus $C\cup \{{\xi}\}$ is $0$-homogeneous.
Since ${\xi}\in B$, we have ${\xi}\in C$ by the maximality of $C$,
which contradicts $B\subs M$.
\end{proof}

\begin{theorem}[Erd\H os--Rado]
$\cp \to (\cp, \oo+1)^2$.
\end{theorem}

\begin{proof}
Fix a function $f:\br \cp;2;\to 2.$

Let $\Sigma$ be a large enough finite set of formulas.
By Corollary \ref{cor:elem}(3) there is a set $M$ with $|M|=\mf c$
such that  $f\in M\prec_\Sigma V$, $M\cap \cp\in
\cp$
and $\br M;{\omega};\subs M$.

Pick ${\xi}\in \cp\setm M$.

Let $A$ be a $\subs$-maximal subset of $M\cap {\kappa}$
such that $A\cup \{{\xi}\}$ is $1$-homogeneous.
If $A$ is uncountable, then we are done.

Assume that $A$ is countable. Since $\br M;{\omega};\subs M$,
we have $A\in M$.

Let
\begin{equation}
B=\{{\beta}\in {\kappa}\setm A:\forall {\alpha}\in A\
f({\beta},{\alpha})=1\}. \end{equation}
Since $f,A\in M$ we have  $B\in M$.
Let $C\subs B$ be a $\subs$-maximal $0$-homogeneous subset.

\noindent{\bf Claim:} {\em $|C|=\cp$}.\\
Assume on the contrary that  $|C|\le \mf c$.
Then $|C|\subs \mf c \subs M$ and so  $C\subs M$
by Claim \ref{cl:card}.
Let ${\gamma}\in C$. Since ${\gamma}\in M\setm A$ we have that
$A\cup \{{\gamma}\}\cup \{{\xi}\}$ is not $1$-homogeneous.
But $A\cup \{{\xi}\}$ is  $1$-homogeneous and ${\gamma}\in B$,
so   $f({\gamma},{\xi})=0$.
Thus $C\cup \{{\xi}\}$ is $0$-homogeneous.
Since ${\xi}\in B$, we have ${\xi}\in C$ by the maximality of $C$,
which contradicts $B\subs M$.
\end{proof}

Given a {\em set-mapping} $F:X\to \mc P(X)$
we say that a subset $Y\subs X$ is  {\em $F$-free}
iff $y'\notin F(y)$ for $y\ne y'\in Y$.

\begin{theorem}
If ${\kappa}=\cf({\kappa})>{\omega}$ and
$F:{\kappa}\to \br {\kappa};<{\omega};$
then there is an  $F$-free subset $C$ of size ${\kappa}$.
\end{theorem}

\begin{proof}
Let $\Sigma$ be a large enough finite set of formulas.
By Corollary \ref{cor:elem}(2) there is a set $M$ with $|M|<{\kappa}$
such that  $F\in M\prec_\Sigma V$ and $M\cap {\kappa}\in {\kappa}$.

Let
${\xi}\in {\kappa}\setm M$ and $A=F({\xi})\cap M$.
Let $C$ be a $\subs$-maximal $F$-free subset of ${\kappa}\setm A$.
Since $F,A\in M$ we can assume that $C\in M$.

\noindent{\bf Claim:} {\em $|C|={\kappa}$}.\\
Assume on the contrary that  $|C|<{\kappa}$. Then $C\subs M$
by Claim \ref{cl:card}. Since $F({\gamma})\subs M$ 
for ${\gamma}\in C$
and
$F({\xi})\cap C\subs  A\cap C=\empt$ we have that
$C\cup \{{\xi}\}$ is also $F$-free.
So $C$ was not $\subs$-maximal. Contradiction.
\end{proof}

First we prove a  weak form of Fodor's Pressing Down Lemma.
A function  $f$ mapping a set of ordinals into the ordinals is   
 called {\em regressive} iff $f(\alpha)<\alpha$ for each $\alpha\in \dom(f)$. 
\begin{theorem}
If ${\kappa}=\cf ({\kappa})>{\omega}$,
$f:{\kappa}\to {\kappa}$ is a regressive function
then there is ${\eta}<{\kappa}$ such that
$f^{-1}\{{\eta}\}$ is unbounded in ${\kappa}$.
\end{theorem}

\begin{proof}
Let $\Sigma$ be a large enough finite set of formulas.
By Corollary \ref{cor:elem}(2) there is a set $M$ with $|M|<{\kappa}$
such that  $f\in M\prec_\Sigma V$ and $M\cap {\kappa}\in {\kappa}$.

Let ${\xi}=\sup (M\cap {\kappa})$ and
consider ${\eta}=f({\xi})$.
We claim that $T=f^{-1}\{{\eta}\}$ is unbounded in ${\kappa}$.
Since ${\eta}\in{\xi}=M\cap {\kappa}$ we have  $T\in N$.
If $T$ is bounded,  then
$\sup T\in M\cap {\kappa}={\xi}$. However ${\xi}\in T$, so $T$
should be unbounded.
\end{proof}

\begin{theorem}(Fodor's Pressing Down Lemma)
If ${\kappa}=\cf ({\kappa})>{\omega}$,
$S\subs {\kappa}$ is stationary, and $f:S\to {\kappa}$ is a regressive
function
then there is an ordinal ${\eta}<{\kappa}$ such that
$f^{-1}\{{\eta}\}$ is stationary.
\end{theorem}

\begin{proof}
Let $\Sigma$ be a large enough finite set of formulas.
By Corollary \ref{cor:elem}(4) there is a set $M$ with $|M|<{\kappa}$
such that  $S,f\in M\prec_\Sigma V$ and ${\xi}=M\cap {\kappa}\in S$.

Let ${\eta}=f({\xi})$.
We show that  $T=f^{-1}\{{\eta}\}$ is stationary.
Clearly $T\in M$. If $T$ is not stationary then there is a closed
unbounded  set $C\in M$ such that $C\cap T=\empt$.

\smallskip

\noindent{\bf Claim:} {\em $\sup (M\cap {\kappa})\in C$ if  $C\in M$
is a closed unbounded subset of ${\kappa}$}.\\
Since $C$ is closed, if $\sup(M\cap {\kappa})\notin C$ then there is 
${\alpha}< \sup(M\cap {\kappa})$ such that $(C\setm {\alpha})\cap
M=\empt$. Then $M\models $ ``$C\setm {\alpha}=\empt$''.
Thus $V\models $ ``$C\setm {\alpha}=\empt$'', i.e. $C\subs
{\alpha}$, which contradicts the assumption that $C$ is unbounded.

\smallskip


\smallskip

So by the claim ${\xi}\in C\cap T$. Contradiction.
\end{proof}

\section{Decomposition theorems}\label{sc:nw}
In the previous section we proved theorems which claimed that 
 {\em ``Given a large enough structure
$\mc A$ we can find a large enough nice substructure of $\mc A$.''}
In this section we prove results which have a different flavor:
{\em Every large structure having certain properties can be
  partitioned into  ``nice'' small pieces.}

In \cite{nw} the following statements were proved:
\begin{theorem}[Nash-Williams]\label{tm:nw}
 $G$ is decomposable into cycles if and only if
it has no odd cut.
\end{theorem}

We give a new proof which illustrates how one can use ``{\em chains of
elementary submodels}''.
To do so we need two lemmas. The first one was proved in 
 section \ref{sc:first}:
\begin{qlemma}
If $G=\<{\kkappa},E\>$ is an NW-graph,
$G\in M\prec_\Sigma  V$ for some large enough finite set $\Sigma$ of
formulas,
then $G\rrestriction M$ is also an NW-graph.
\end{qlemma}
The second one is  the following statement.
\begin{lemma}\label{lm:nw2}
If $G=\<{\kkappa},E\>$ is an NW-graph,
$G\in M\prec_\Sigma  V$ for some large enough finite set $\Sigma$ of
formulas,
then $G\ssetm M$ is also an NW-graph.
\end{lemma}

Lemma \ref{lm:nw2} above follows easily from the next  one.

\begin{lemma}\label{lm:GminuszM}
Assume that  $M\prec_\Sigma V$   with $|M|\subs M$ for some
large enough  finite set $\Sigma$ of
formulas. If $G
\in M$ is a graph,
$x\ne y\in V(G)$ and  $F\subs E(G\ssetm M)$,
such that
\begin{equation}
\text{$|F|\le |M|$, ${\gamma}_{G\ssetm M }(x,y)>0$ and  $F$ separates
$x$ and $y$ in $G\ssetm M$}
\end{equation}
then
\begin{equation}
\text{$F$ separates $x$ and $y$
in $G$.}
\end{equation}
\end{lemma}

\begin{proof}[Proof of Lemma \ref{lm:nw2} from Lemma \ref{lm:GminuszM}]
Assume on the contrary that $G\ssetm M$ has an odd cut $F$.
Since any cut is the disjoint union of bonds we can assume that
$F$ is a bond.

Pick $c_1c_2\in F$. Then clearly ${\gamma}_{G\ssetm M}(c_1,c_2)>0$.
Moreover $F$ separates $c_1$ and $c_2$ in $G\ssetm M$, so 
$F$ separates them in $G$ by  Lemma \ref{lm:GminuszM},
i.e. $c_1$ and $c_2$ are in different connected components of $G\ssetm F$ 

However $F$ can not be a bond in $G$, so by Proposition 
\ref{pr:bond} there is a connected component
$D$ of $G\ssetm F$ such that $F\subs \br D;2;$.
i.e. $c_1$ and $c_2$ are in the same connected component of $G\ssetm F$.
This contradiction proves the lemma.
\end{proof}

\begin{proof}[Proof of Lemma \ref{lm:GminuszM}]
Assume that $G$, $M$, $x$, $y$ and
$F$ form a counterexample.

\begin{figure}[h]
\psfrag*{M}{$M$}
\psfrag*{F}{$F$}
\psfrag*{M}{$M$}
\psfrag*{x}{$x$}
\psfrag*{y}{$y$}
\psfrag*{Q}{$Q$}
\psfrag*{qx}{$Q_x$}
\psfrag*{qy}{$Q_y$}
\psfrag*{P}{$P$}
\psfrag*{T}{$T$}
\psfrag*{R}{$R$}
\psfrag*{G}{$G$}
\psfrag*{qjx}{$q_{j_x}$}
\psfrag*{qjy}{$q_{j_y}$}
\includegraphics[height=5cm]{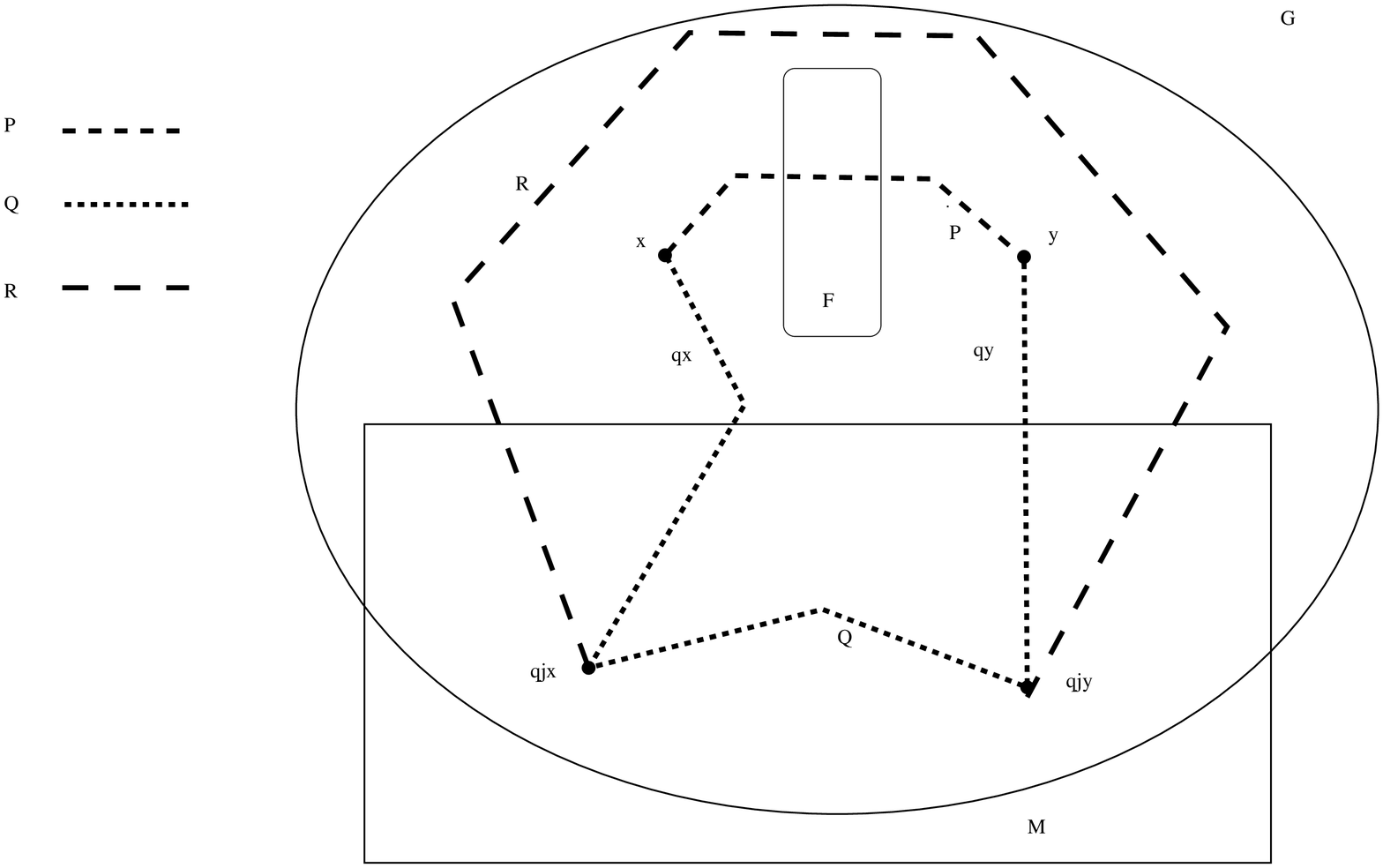}
\caption{}
\end{figure}

Fix a path $P=p_0p_1\dots p_n$ from $x$ to $y$ in $G\ssetm M$ which
witnesses that ${\gamma}_{G\ssetm M}(x,y)>0$, i.e.
 $p_0=x$, $p_n=y$ and $p_ip_{i+1}\in E(G)\setm M$ for $i<n$.

We assumed that $F$ does not separate $x$ and $y$ in $G$,
so there is a path $Q=q_0\dots q_m$ from $x$ to $y$
witnessing this fact, i.e. $q_0=x$, $q_m=y$ and 
$q_jq_{j+1}\in E(G)\setm F$   for $j<m$. Since
$F$ separates $x$ and $y$ in $G\ssetm M$ there is at least
one $j^*<m$ such that $q_{j^*}q_{j^*+1}\in M$.

Let $j_x=\min \{j:q_j\in M\}$ and $j_y=\max\{j:g_j\in M\}$.
Since $j_x\le j^*$ and $j_y\ge j^*+1$ we have $j_x<j_y$.
Let $x'=q_{j_x}$ and $y'=q_{j_y}$.
Let $Q_x=q_{j_x}q_{j_x-1}\dots q_1q_0$
and $Q_y=q_mq_{m-1}\dots q_{j_y}$.
Then $Q_xPQ_y$ is a walk from $x'$ to $y'$ in $G\ssetm M$.
Hence ${\gamma}_{G\ssetm M}(x',y')>0$. 

\smallskip

\noindent{\bf Claim:} {\em ${\gamma}_{G}(x',y')>|M|$.}

Indeed, assume that 
${\lambda}={\gamma}_{G}(x',y')\le |M|$.
Since $M\prec_\Sigma V$ and 
$x',y'\in M$  there is $A\in M\cap \br V(G);{\lambda};$ 
such that $A$ separates $x'$ and $y'$ in $G$.
Since $|A|={\lambda}\subs M$ we have $A\subs M$.
So $M$ separates $x'$ and $y'$, i.e. ${\gamma}_{G\ssetm M}(x',y')=0$.
This contradiction proves the claim.

\medskip

By the weak Erd\H os-Menger Theorem  there are 
${\gamma}_G(x',y')$ many edge disjoint paths between 
$x'$ and $y'$ in $G$.
Since $|M\cup F|=|M|< {\gamma}_{G}(x',y')$ there is a path $R=r_0\dots r_k$ from $x'$ to $y'$ which avoids
$M\cup F$.
Then $Q_x^{-1}RQ_y^{-1}$ is walk from $x$ to $y$ in $G\ssetm M$
which avoids $F$. Contradiction.
\end{proof}

\begin{proof}[Proof of theorem \ref{tm:nw}]
We prove the theorem by induction on $|V(G)|$.

If $G$ is countably  infinite then
for each $e\in E(G)$ there is a cycle $C$ in $G$
with $e\in E(C)$ because $e$ is not a cut in $G$. 
Moreover,  $G\ssetm C$ is also an NW-graph, i.e.
it does not have odd cuts.  Using this observation
we can construct a sequence
$\{C_i:i<{\omega}\}$ of edge disjoint cycles in $G$
with $E(G)=\cup\{E(C_i):i<{\omega}\}$.

Assume now that ${\kappa}=V(G)>{\omega}$
and we have proved the statement for graphs of cardinality $<{\kappa}$.

Let $\Sigma$ be a large enough finite set of formulas.
By the Reflection Principle \ref{pr:refl}
there is a  cardinal ${\lambda}$ such that
  $V_{\lambda}\prec_\Sigma V$
and $\br V_{\lambda};{\kappa};\subs V_{\lambda}$.
Then $G\in V_{\lambda}$.

We will construct a sequence
$\<M_{\alpha}:{\alpha}<{\kappa}\>\subs V_{\lambda}$
of elementary submodels of
$V_{\lambda}$ with
\begin{equation}
\tag{$*_{\alpha}$}
\text{$|M_{\alpha}|= {\omega}+|{\alpha}|$, ${\alpha}\subs M_{\alpha}$
and $M_{\alpha}\in M_{{\alpha}+1}$}
\end{equation}
 as follows:
\begin{enumerate}[(i)]
\item Let $M_0$ be a countable elementary submodel
of $V_{\lambda}$ with $G\in M_0$.
\item if ${\beta}<{\kappa}$ is a limit then let
$M_{\beta}=\cup\{M_{\alpha}:{\alpha}<{\beta}\}$.
Since $|M_{\beta}|\le {\omega}+|{\beta}|<{\kappa}$ and
$M_{\beta}\subs V_{\lambda}$ we have $M_{\beta}\in V_{\lambda}$.
\item If ${\beta}={\alpha}+1$ then
$|M_{\alpha}\cup \{M_{\alpha}\}\cup  {\beta}|= {\omega}+|{\beta}|$
so by L\"owenhein-Skolem Theorem
there is $M_{\beta}\prec V_{\lambda}$ with
$M_{\alpha}\cup \{M_{\alpha}\}\cup  {\beta}\subs M_{\beta} $ and
$|M_{\beta}|={\omega}+|{\beta}|$.
\end{enumerate}
The construction clearly guarantees $(*_{\alpha})$.
Using the chain $\<M_{\alpha}:{\alpha}<{\kappa}\>$
decompose $G$ as follows:
\begin{itemize}
\item for ${\alpha}<{\kappa}$ let $G_{\alpha}=(G\ssetm
  M_{\alpha})\rrestriction {M_{{\alpha}+1}}$.
\end{itemize}
By Lemma \ref{lm:nw2} the graph $G_{\alpha}'=G\ssetm M_{\alpha}$ is
NW. Moreover, since $M_{\alpha}\in M_{{\alpha}+1}$ we have
$G\ssetm M_{\alpha}\in M_{{\alpha}+1}$. So we can apply
Lemma \ref{lm:nw1} for $M_{{\alpha}+1}$ and $G_{\alpha}'$ to deduce
that $G_{\alpha}$ is NW.

So we have decomposed the graph $G$ into NW-graphs
$\{G_{\alpha}:{\alpha}<{\kappa}\}$.  Moreover,
$|V(G_{\alpha})|\le |M_{{\alpha}+1}|\le {\omega}+|{\alpha}|<{\kappa}$,
so by the inductive hypothesis, every $G_{\alpha}$ is the union of
disjoint cycles. So $G$ itself is the union of disjoint cycles which was
to be proved.
\end{proof}

\subsection{General framework}

If $\Phi$ is a graph property then we write $G\in \Phi$
to mean that the graph $G$ has property $\Phi$.

We say that a graph property $\Phi$ is {\em well-reflecting}
iff for each graph $G\in \Phi$ whenever $G\in M \prec_{\Sigma}V$
with $|M|\subs M$ for some large enough finite set $\Sigma$ of formulas,
we have both $G\rrestriction M\in \Phi$ and $G\ssetm M\in \Phi$.

\begin{theorem}\label{tm:r1}
 Let $\Phi$ be a well-reflecting graph property.
Then every graph $G\in \Phi$ can be decomposed into
a family $\{G_i:i \in I\}\subs \Phi$ of countable graphs.
\end{theorem}

To prove this theorem we need to introduce the following notion. 
Let ${\kappa}$ and ${\lambda}$ be cardinals.
We say that
$\<M_{\alpha}:{\alpha}<{\kappa}\>$ is a
{\em ${\kappa}$-chain of submodels of
  $V_{\lambda}$}
iff
\begin{enumerate}
\item the sequence $\<M_{\alpha}:{\alpha}<{\kappa}\>\subs
  V_{\lambda}\cap \br
V_{\lambda};<{\kappa};$ is strictly increasing and continuous
(i.e. $M_{\beta}=\cup\{M_{\alpha}:{\alpha}<{\beta}\}$ for limit ${\beta}$),
\item
  $M_{\alpha}\prec V_{\lambda}$, ${\alpha}\subs M_{\alpha} $
and $M_{\alpha}\in M_{{\alpha}+1}$ for ${\alpha}<{\kappa}$,
\end{enumerate}

\begin{fact}\label{fa:chain}
If   $\br V_{\lambda};<{\kappa};\subs V_{\lambda}$ then
for each $x\in V_{\lambda}$ there is a
${\kappa}$-chain of elementary submodels
$\<M_{\alpha}:{\alpha}<{\kappa}\>$ of
$V_{\lambda}$ with $x\in M_0$ and ${\alpha}\subs M_{\alpha}$
for ${\alpha}<\kappa$.
\end{fact}

\begin{proof}
Actually such a chain was constructed  in the proof of
Theorem \ref{tm:nw}.
\end{proof}

\begin{proof}[Proof of Theorem \ref{tm:r1}]
By induction on $|G|$.
If $|G|$ is countable then there is nothing to prove.

Assume that 
$G=\<{\kappa},E\>$ and 
${\kappa}>{\omega}$.
By the Reflection Principle \ref{pr:refl}
there is a cardinal ${\lambda}$ such that $V_{\lambda}\prec_\Sigma V$
and $\br V_{\lambda};{\kappa};\subs V_{\lambda}$.
Then, by  Fact \ref{fa:chain} there is a
${\kappa}$-chain of elementary submodels
of $V_{\lambda}$ with $G\in M_0$.
 For ${\alpha}<{\kappa}$ let $G_{\alpha}=(G\ssetm
  M_{\alpha})\rrestriction {M_{{\alpha}+1}}$.
 Since $\Phi$ is well-reflecting,
the graph $G_{\alpha}'=G\ssetm M_{\alpha}$ is in $ \Phi$.
 Moreover, since $M_{\alpha}\in M_{{\alpha}+1}$ we have
$G\ssetm M_{\alpha}\in M_{{\alpha}+1}$. 
So applying once more the fact that $\Phi$ is well reflecting
 for $M_{{\alpha}+1}$ and $G_{\alpha}'$ we obtain that 
 $G_{\alpha}$ is in $\Phi$.

So we have decomposed the graph $G$ into graphs
$\{G_{\alpha}:{\alpha}<{\kappa}\}\subs \Phi$.  However
$|V(G_{\alpha})|\le |M_{{\alpha}+1}|\le {\omega}+|{\alpha}|<{\kappa}$,
so by the inductive hypothesis, every $G_{\alpha}$ has a decomposition  
 $\mc G_{\alpha}$ into countable elements of $\Phi$.
Then $\mc G=\cup\{\mc G_{\alpha}:{\alpha}<{\kappa}\}$ is the desired
decomposition of $G$.
\end{proof}

\begin{theorem}\label{tm:rmain}
Let $\Phi$ and $\Psi$ be graph properties.
Assume that
\begin{enumerate}[(1)]
\item $\Phi$ is well-reflecting,
\item if $H\in \Phi$ is a countable graph then $H\in \Psi$,
\item if $G$ has a decomposition
$\{G_i:i\in I\}$ with $G_i\in \Psi$ then $G\in \Psi$.
\end{enumerate}
Then $G\in \Phi$ implies $G\in \Psi$.
\end{theorem}

\begin{proof}
Theorem \ref{tm:r1}
and (1) yield that $G$ 
has a decomposition into 
 countable graphs $\{G_i:i\in I\}\subs
\Phi$. By $(2)$ , $\{G_i:i\in I\}\subs
\Psi$. Finally, by (3), this implies $G\in \Psi$ which was to be proved.
\end{proof}

In Lemmas \ref{lm:nw1} and \ref{lm:nw2}
we proved that the graph property
``{\em there is  no odd cut}'' is  well-reflecting.

As we will see, Theorem \ref{tm:rmain} can be applied as
a ``black box'' principle in many proofs.

\subsection{Applications of Theorem \ref{tm:rmain}}
First we give a new proof of   a result of Laviolette.

\begin{theorem}[{\cite[Corollary 1]{La}}]\label{tm:bridgeless}
Every bridgeless graph can be partitioned into countable bridgeless
graphs.
\end{theorem}

\begin{proof}
We need the following lemma:  

\begin{lemma}\label{lm:bridgeless-wr}\label{tm:la}
The ``{\em bridgeless}'' property  is well-reflecting.  
\end{lemma}

\begin{proof}[Proof of Lemma \ref{lm:bridgeless-wr}]
Assume that $G$ is a  graph and  $G\in M\prec_\Sigma V$ for
some large enough finite family $\Sigma$ of formulas.\\
(1)  Assume that an edge $e=xy$ is a bridge in $G\rrestriction M$.
Then 
\begin{equation}
M\models \text{$e$ separates $x$ and $y$},  
\end{equation}
so, by $M\prec_\Sigma V$
\begin{equation}
V\models \text{$e$ separates $x$ and $y$},  
\end{equation}
i.e. $e$ is a bridge in $G$.\\
(2) Assume that an edge $e=xy$ is a bridge in $G\ssetm  M$.
Then $e$ separates $x$ and $y$ in $G\ssetm M$,
so by Lemma \ref{lm:GminuszM}, $e$ separates $x$ and $y$ in $G$,
i.e. $e$ is a bridge in $G$. 
\end{proof}

By Lemma \ref{lm:bridgeless-wr}, we can apply Theorem \ref{tm:r1}
to get the statement of this theorem.
\end{proof}

Let us formulate two corollaries.
\begin{corollary}[Laviolette, {\cite[Theorem 1]{La}}]\label{co:la}
Every bridgeless graph has a cycle ${\omega}$-cover.
\end{corollary}

\begin{proof}
Every countable bridgeless graph clearly has a  cycle ${\omega}$-cover,
and by the previous theorem every bridgeless graph can be partitioned
into countable bridgeless graphs. 
\end{proof}

It is worth mentioning that 
in \cite{La} Theorem \ref{tm:la} was a corollary of 
Corollary \ref{co:la}.

Before formulation of the second corollary 
let us recall the following conjecture of Seymour and Szekeres.
\begin{dcc}
Every bridgeless graph has a cycle double cover.
\end{dcc}

Since every bridgeless graph can be partitioned into countable
bridgeless graphs, we yield
\begin{corollary}[Laviolette, \cite{La}]
If the  Double Circle Conjecture holds for all countable
graphs then it holds for all graphs.
\end{corollary}

Next we sketch two more applications.

In \cite{nw} the following statements were also proved:
\begin{theorem}[Nash-Williams]\label{tm:nw2}
(1) A graph $G$ can be decomposed into cycles and
endless chains  if and only if it has no
vertex of odd valency. (2) $G$ is decomposable into endless chains if
and only if it has no vertex of odd valency and no  finite
non-trivial component.
\end{theorem}
Let us recall that a connected component is {\em non-trivial}
if it has at least two elements.

\begin{proof}[Proof of \ref{tm:nw2}]
For $j=1,2$
we say that a graph $G$ is {\em $NW_j$} iff
$G$ satisfies the assumption of statement (j) from \ref{tm:nw2}.

\begin{lemma}\label{lm:nw12c} 
The statements of Theorem \ref{tm:nw2} hold for countable graphs.
\end{lemma}

The proof   of Lemma
\ref{lm:nw12c} is left to the reader.

\begin{lemma}\label{lm:nw12}
The following graph properties are well-reflecting:
\begin{enumerate}[(1)]
\item  there is no vertex of odd valency.
\item  there is no finite non-trivial component.
\end{enumerate}
\end{lemma}

\begin{proof}[Proof of Lemma \ref{lm:nw12}]
(1)
Assume that in $G$ there is no vertex of odd valency.
Let  $G\in M \prec_{\Sigma}V$
with $|M|\subs M$ for some large enough finite set $\Sigma$ of
formulas.

\noindent
{\bf Claim} {\em There is no vertex of odd valency in $G\rrestriction M$.}

Indeed, let $x\in V(G\rrestriction M)=V\cap M$ be arbitrary , and assume that
the set $A=\{v\in V(G \rrestriction M): vx\in E(G\rrestriction M)\}$ is finite.
Since $A\subs M$, we have $A\in M$ by Claim \ref{cl:omega}, and 
for each $v\in V(G \rrestriction M)$ we have $v\in A$ iff $vx\in E(G)\cap M$. 
Thus
\begin{equation}
M\models A=\{v\in V(G):vx\in E(G)\},  
\end{equation}
so, by $M\prec_\Sigma V$, we have 
\begin{equation}
V\models A=\{v\in V(G):vx\in E(G)\},  
\end{equation}
i.e. $A=\{v\in V(G):vx\in E(G)\}$.
Thus  $d_G(x)=d_{G\rrestriction M}(x)$, which proves the claim.

\noindent
{\bf Claim}  {\em There is no vertex of odd valency in $G\ssetm M$.}

Let  $x\in V$ be arbitrary.
If $x\notin M$, then $G(x)=(G\ssetm M) (x)$ because 
$E(G)\setm E(G\setm M)\subs \br M;2;\subs M$, so 
$d_{G\ssetm   M}(x)=d_G(x)$ can not be odd.

Assume  $x\in M $. If $d_G(x)\le |M|$ then 
$\{v\in V(G):vx\in E(G)\}\in M$ implies 
$\{v\in V(G):vx\in E(G)\}\subs M$ by Claim \ref{cl:card} because
$|M|\subs M$, and so 
$d_{G\ssetm M}(x)=0$.
If $d_G(x)> |M|$ then $d_G(x)=d_{G\ssetm M}(x)$.
So $d_{G\ssetm M}(x)$ can not be an odd natural number.

\medskip

\noindent (2)
Assume that in $G$ there is no finite component.
Let  $G\in M \prec_{\Sigma}V$
with $|M|\subs M$ for some large enough finite set $\Sigma$ of
formulas.

\noindent
{\bf Claim} {\em There is no finite non-trivial component in $G\rrestriction M$.}

Let $x\in V(G)\cap M$ and assume that 
$x$ has a finite component $C$ in $G\rrestriction M$.
Then $C\in M$ and 
\begin{equation}
M\models \text{$C$ is the  component of $x$},  
\end{equation}
so 
\begin{equation}
V\models \text{$C$ is the  component of $x$},  
\end{equation}
i.e. $G$ has finite component.

\noindent
{\bf Claim} {\em There is no finite non-empty component in $G\ssetm M$.}

Assume that there is a finite non-trivial component  $C$ in $G\ssetm M$.
Since $C$ is not a component in $M$  there is an edge
$cd\in E(G)\cap M$ with $c\in C$.
Since $C$ is non-trivial there is $c'\in C$ such that  
$cc'$ is an edge in $G\ssetm M$.
Then $c\in M$ and $c'\notin M$.

Since $d_G(c)\le |M|$ would imply 
$c'\in \{c^*:cc^*\in E(G)\}\subs M$ we have 
$d_G(x)>|M|$. However $\{c^*:cc^*\in E(G)\}\setm M\subs C $, 
and so $|C|>|M|$. Contradiction.  
\end{proof}

We want to  apply Theorem \ref{tm:rmain}.
Let $\Phi_i$ be the property NW${}_i$ for $i=1,2$, and 
$\Psi_1$ be  ``{\em decomposable into cycles and
endless chains }'', and
$\Psi_2$ be ``{\em decomposable into 
endless chains }''.
 
Then  condition \ref{tm:rmain}.(1)
holds by Lemma \ref{lm:nw12}, \ref{tm:rmain}.(2) is true by 
Lemma \ref{lm:nw12c}. \ref{tm:rmain}.(3) is trivial from the definition.
Putting these things together   we obtain the theorem.
\end{proof}

\section{Bond faithful decompositions}\label{sc:bond}

In this section we prove a decomposition theorem 
in which we can not apply  Theorem \ref{tm:rmain}.

\begin{definition}\label{df:fa} Let ${\kappa}$ be an infinite cardinal.
A decomposition $\mc H$ of a graph $G$ is {\em ${\kappa}$-bond faithful}
iff  $|E(H)|\le {\kappa}$ for each $H\in \mc H$,
\begin{enumerate}[(i)]
\item any bond of $G$ of cardinality $\le {\kappa}$ is contained in
  some member of the decomposition,
\item any bond of cardinality $<{\kappa}$ of a member of the
  decomposition is a bond of $G$.
\end{enumerate}
\end{definition}

\begin{theorem}[Laviolette, {\cite[Theorem 3]{La}}]
Every graph has a bond-faithful $\omega$-decomposition, and with the assumption
of GCH, every graph has a bond-faithful $\kappa$-decomposition 
for any infinite cardinal 
$\kappa$.
\end{theorem}

Applying methods of elementary submodels leads more naturally 
to a simpler proof of the theorem above that does not rely on GCH. 

\begin{theorem}\label{tm:fa}
For any cardinal ${\kappa}$
every  graph has a ${\kappa}$-bond faithful decomposition.   
\end{theorem}


The following lemma is the key to the proof.
\begin{lemma}\label{lm:bf}
 Let $G$ be a graph, $G\in M \prec_{\Sigma}V$
with $\mu=|M|\subs M$ for some large enough finite set $\Sigma$ of
formulas.
\begin{enumerate}[(I)]
\item If $F\subs E(G\rrestriction M)$ is a bond of $G\rrestriction M$ with
$|F|<|M|$ then $F$ is a bond in $G$.
\item If $F\subs E(G)$ is a bond of $G\ssetm M$ with
$|F|<|M|$ then $F$ is a bond in $G$.
\end{enumerate}
\end{lemma}

\begin{proof}[Proof of \ref{lm:bf}]
(I) 
Assume on the contrary that $F$ is not a bond in $G$.
Pick $xx'\in F$.
Then by Proposition \ref{pr:bond}
$x$ and $x'$ are in the same connected component $D$ of $G\ssetm F$,
and so 
 there is a path 
$P=x_1x_2\dots x_n$, in $G\ssetm F$, $x_1=x$, $x_n=x'$.  
Choose the path  in such a way that the cardinality of the
finite set 
\begin{equation}
I_P=\{i: x_ix_{i+1}\notin M\}   
\end{equation}
is minimal. Since $F$ is a cut in $G\rrestriction M$ we have $I_P\ne
\empt$.
Let $i=\min I_p$. Then  $x_i\in M$.
Let $j=\min\{j>i:x_j\in M\}$.
Then $j>i+1$, $x_i,x_j\in M$, and moreover  
${\gamma}_{(G\ssetm M)\ssetm F}(x_i,x_j)>0$.

\begin{claim}\label{cl:gmkappa}
If $x,y\in M$, ${\gamma}_{G\ssetm M}(x,y)>0$ then 
${\gamma}_{G\rrestriction M}(x,y)=|M|$.  
\end{claim}

\begin{proof}[Proof of the Claim]
There is a vertex set $A\in \br V(G);{{\gamma}_G(x,y)};$
such that $A$ separates $x$ and $y$ in $G$.
We assumed that 
$\Sigma$ is large enough, especially it contains the formula
$\exists A\ \varphi(A,x,y,G)$, where
$\varphi(A,x,y,G)$ is the following formula:
\begin{multline}\notag
\text{\em $A\in \br V(G);{{\gamma}_G(x,y)};$ is a vertex set
which  separates $x$ and $y$ in $G$}. 
\end{multline}
Since $M\prec_\Sigma V$, and the parameters $G,x,y$ are in $M$, 
there is   an $A$ in $M$ such that 
$M\models \varphi(A,x,y,G)$.
Since we assumed that 
$\Sigma$ is large enough,  it contains the formula
$\varphi(A,x,y,G)$.
So $V\models \varphi(A,x,y,G)$, i.e. 
$A\in \br V(G);{{\gamma}_G(x,y)};\cap M$ is a vertex set
which  separates $x$ and $y$ in $G$.

If ${\gamma}_G(x,y)\le {\mu}\subs M$ then  
$A\in M$ implies $A\subs M$   by Claim \ref{cl:card}, and so $M$ separates $x$ and
$y$ in $G$. Thus ${\gamma}_{G\ssetm M}(x,y)=0$.

But ${\gamma}_{G\ssetm M}(x,y)>0$, so we have 
${\gamma}_G(x,y)>|M|$.
So, by the weak Erd\H os-Menger Theorem  there is a family
$\mc P$ of  
${\mu}$ many edge disjoint paths between 
$x$ and $y$ in $G$.
Since $G,x,y,{\mu}\in M$ we can find such a $\mc P$ in $M$.
But $|\mc P|={\mu}\subs M$, and so $\mc P\subs M$.
Thus there are ${\mu}$-many 
edge disjoint paths between 
$x$ and $y$ in $M$, i.e. 
${\gamma}_{G\rrestriction M}(x,y)={\mu}$. 
\end{proof}

By the Claim  ${\gamma}_{G\rrestriction M}(x_i,x_j)={\mu}$.
So,   by the weak infinite Menger Theorem, 
there are ${\mu}$ many edge disjoint path in 
$G\rrestriction M$ between $x_i$ and $x_j$.
Since $|F|<{\mu}$, there is a path 
$Q=x_iy_1\dots y_kx_j$ in $G\rrestriction M$ which  avoid
$F$. Then 
$P'=x_1\dots x_jy_1\dots y_kx_j\dots x_n$
is a path between $x_1$ and $x_n$ in $G\ssetm F$ with $|I_{P'}|<|I_P|$. Contradiction.

\noindent (II)
Let $c_1c_2\in F$.
Then ${\gamma}_{G\ssetm M}(c_1,c_2)>0$, $F$ separates
  $c_1$ and $c_2$ in $G\ssetm M$, so $F$ also separates  
$c_1$ and $c_2$ in $G$ by Lemma \ref{lm:GminuszM}.
In other words,  $c_1$ and $c_2$ are in different connected component of
$G\ssetm F$, and so 
$F$ should be a bond in $G$ by Proposition \ref{pr:bond}. 
\end{proof}

\begin{proof}[Proof of Theorem \ref{tm:fa}]
By induction on $|V(G)|$.
If $|V(G)|\le {\kappa}$ then the  one element decomposition
$\{G\}$
 works.

Assume that 
$G=\<{\mu},E\>$, and 
${\mu}>{\kappa}$.
Let $\Sigma$ be a large enough finite set of formulas.
By the Reflection Principle \ref{pr:refl}
there is a  cardinal ${\lambda}$ such that
  $V_{\lambda}\prec_\Sigma V$
and $\br V_{\lambda};{\mu};\subs V_{\lambda}$.

By Fact \ref{fa:chain} there is a 
${\mu}$-chain of elementary submodels
$\<M_{\alpha}:{\alpha}<{\mu}\>$ of
$V_{\lambda}$ with $G,{\kappa}\in M_0$.
Since and ${\kappa}<{\mu}$ and ${\alpha}\subs M_{\alpha}$ for ${\alpha}<{\mu}$, 
we can assume that ${\kappa}\subs M_0$.

Using the chain $\<M_{\alpha}:{\alpha}<{\mu}\>$
partition $G$ as follows:
\begin{itemize}
\item for ${\alpha}<{\mu}$ let $G_{\alpha}=(G\ssetm
  M_{\alpha})\rrestriction {M_{{\alpha}+1}}$.
\end{itemize}
Let $G_{\alpha}'=G\ssetm M_{\alpha}$.
By Lemma \ref{lm:bf}(II) 
\begin{itemize}
\item any bond of cardinality $<{\kappa}$ of $G'_{\alpha}$ is a bond of $G$.
\end{itemize}
Moreover, since $M_{\alpha}\in M_{{\alpha}+1}$ we have
$G\ssetm M_{\alpha}\in M_{{\alpha}+1}$. So we can apply
Lemma \ref{lm:bf}(I) for $M_{{\alpha}+1}$ and $G_{\alpha}'$ to derive
that 
\begin{itemize}
\item any bond of cardinality $<{\kappa}$ of $G_{\alpha}$ is a bond of $G'_{\alpha}$.
\end{itemize}
Putting these together
\begin{itemize}
\item[$\circ$] any bond of cardinality $<{\kappa}$ of $G_{\alpha}$ is a bond of $G$.
\end{itemize}
Moreover
$|V(G_{\alpha})|\le |M_{{\alpha}+1}|\le {\omega}+|{\alpha}|<{\mu}$,
so by the inductive hypothesis, every $G_{\alpha}$ 
has a ${\kappa}$-bond faithful decomposition $\mc H_{\alpha}$.
Let  $\mc H=\cup\{\mc H_{\alpha}:{\alpha}<{\mu}\}$.
 $\mc H$ clearly satisfies \ref{df:fa}(ii):
if $F$ is a bond of some $H\in \mc H_{\alpha}$ with $|F|<{\kappa}$,
then $F$ is a bond of $G_{\alpha}$, and so $F$ is a bond of $G$
by ($\circ$).

Finally we show that $\mc H$ satisfies  \ref{df:fa}(i) as well.
We recall one more result of Laviolette:
\begin{theorem}[{\cite[Proposition 3]{La}}]
For any cardinal ${\kappa}$
every  graph has a decomposition  $\mc K$ which satisfies 
\ref{df:fa}(i) and $|E(K)|\le {\kappa}$ for each $K\in \mc K$.
  \end{theorem}
Let us remark that GCH was assumed in \cite[Proposition 3]{La},
but in the proof it was not used.

Let $\varphi(G',{\kappa}',\mc K')$ be the following formula:
\begin{multline}\notag
\text{\em $\mc K'$ is a decomposition of $G'$ which satisfies 
\ref{df:fa}(i)}\\
\text{\em and $|E(K)|\le {\kappa}$ for each $K'\in \mc K'$.} 
\end{multline}
Since $\Sigma$ was ``large enough'' we can assume that 
it contains the formulas $\varphi(G',{\kappa}',\mc K')$
and $\exists \mc K' \varphi(G',{\kappa}',\mc K')$.
Since $M_0\prec_\Sigma V$, and $G,{\kappa}\in M_0$ we have 
a $\mc K\in M_0$ such that $\varphi(G,\kappa,\mc K)$ holds, i.e.
$\mc K$ is a decomposition of $G$ which witnesses \ref{df:fa}(i) and
$|E(K)|\le {\kappa}$ for each $K\in \mc K$.
Assume that $A$ is a bond of $G$ with $|A|\le {\kappa}$.
Then there is $K\in \mc K$ such that $A\subs E(K)$. 
Let ${\alpha}$ be minimal such that $
E(K)\cap M_{{\alpha}+1}\ne \empt$, and pick $e\in E(K)\cap M_{{\alpha}+1}$.
Then 
$K$ is definable from the parameters $\mc K, e\in M_{{\alpha}+1}$
by the  formula ``$K\in \mc K\land e\in K$''.
So $K\in M_{{\alpha}+1}$ by Claim \ref{cl:def}. Thus $A\subs E(K) \subs E(G_{\alpha})$.
Since, by the inductive assumption, 
the decomposition $\mc H_{\alpha}$ satisfies \ref{df:fa}(i)
there is $H\in \mc H_{\alpha}$
with $A\subs E(H)$. But $H\in \mc H$, so we are done.
\end{proof}

\begin{thebibliography}{99}


\bibitem{AbMa} 
	
Abraham,
Uri;  Magidor, Menachem; {\em Cardinal Arithmetic} in
Handbook of Set Theory, Eds. Foreman, Kanamori, 2010.

\bibitem{BHT}
Baumgartner, James; Hajnal, András; Todorčević, Stevo
{\em Extensions of the Erdős-Rado theorem. }
Finite and infinite combinatorics in sets and logic (Banff, AB, 1991), 1–17,
NATO Adv. Sci. Inst. Ser. C Math. Phys. Sci., 411, Kluwer Acad. Publ., Dordrecht, 1993. 


\bibitem{CK}
Chang, C. C.; Keisler, H. J.
{\em Model theory}.
Second edition. Studies in Logic and the Foundations of Mathematics, 73. North-Holland Publishing Co., 
Amsterdam-New York-Oxford, 1977. xii+554 pp.


\bibitem{Di} 
Diestel, Reinhard  {\em Graph theory.}
Graduate Texts in Mathematics, 173. Springer-Verlag, New York,  2000.


\bibitem{Dow}Dow, Alan  {\em An introduction to applications of
  elementary submodels to  topology.}
 Topology Proc.  13  (1988),  no. 1, 17--72.



\bibitem{Geschke} Geschke, Stefan. {\em Applications of elementary 
submodels in general topology.}
Foundations of the formal sciences, 1 (Berlin, 1999).
 Synthese  133  (2002),  no. 1-2, 31--41


\bibitem{Hajnal} Hajnal, And\'as {\em On the Chromatic Number of
  Graphs and Set Systems} PIMS Distinguished Chair Lectures, 2004.

\bibitem{HJSSz} Hajnal, And\'as; Juhász, István; Soukup, Lajos; Szentmiklóssy,
Zoltán, {\em Conflict free colorings of (strongly) almost disjoint set-systems }, Acta Math. Hung., 
to appear. 
arXiv:1004.0181v1

\bibitem{Jech} Jech, Thomas  {\em Set theory.
The third millennium edition, revised and expanded.}
Springer Monographs in Mathematics. Springer-Verlag, Berlin,  2003.



\bibitem{Just-Weese}
Just, Winfried; Weese, Martin 
{\em Discovering modern set theory. II. 
Set-theoretic tools for every mathematician.} 
Graduate Studies in Mathematics, 18. 
American Mathematical Society, Providence, RI, 1997.


\bibitem{KoSh}
Komjáth, Péter ;  Shelah, Saharon . 
{\em Forcing constructions for uncountably chromatic graphs.}
 J. Symbolic Logic  53  (1988),  no. 3, 696--707.


\bibitem{Kunen}  Kunen, Kenneth  {\em Set theory.
An introduction to independence proofs.}
Studies in Logic and the Foundations of Mathematics, 102. North-Holland Publishing Co., Amsterdam-New York,  1980.


\bibitem{La} 
Laviolette, François {\em Decompositions of infinite
  graphs. I. Bond-faithful decompositions.}
  J. Combin. Theory Ser. B  94  (2005),  no. 2, 259--277.

\bibitem{nw} Nash-Williams, C. St. J. A.  
{\em Decomposition of graphs into closed and endless chains.}
 Proc. London Math. Soc. (3)  10  1960 221--238.

\bibitem{sh:pcf} 
Shelah, Saharon . {\em Cardinal arithmetic.}
Oxford Logic Guides, 29. Oxford Science Publications.
The Clarendon Press, Oxford University Press, New York,  1994.
%



\bibitem{S}  Simpson, Stephen G., {\em Model theoretic proof a partition theorem},
Abstracts of Contributed Papers, 
Notices of AMS, 17 (1970), no 6  p 964.


\end{thebibliography}
\end{document}